\documentclass[preprint,12pt]{elsarticle}

\usepackage{graphicx}
\usepackage{amssymb}
\usepackage{amsmath}
\usepackage{mathtools}
\usepackage{seqsplit}
\usepackage[utf8]{inputenc}

\usepackage{lineno}
\usepackage{amsthm}
\usepackage{dsfont}

\newcommand{\SL}{\operatorname{SL}}

\newcommand{\tr}{\operatorname{tr}}
\newcommand{\Rex}{\operatorname{Re}}

\newcommand{\Mod}[1]{\ (\mathrm{mod}\ #1)}

\newtheorem{lemma}{Lemma}

\newtheorem{corollary}{Corollary}

\newtheorem{proposition}{Proposition}
\journal{Journal Name}

\begin{document}

\begin{frontmatter}


\title{Real irreducible representations of $\SL_2(q)$ and their fixed point dimensions for cyclic subgroups}



\author{Piotr Mizerka}
\address{Adam Mickiewicz University in Pozna\'n, Poland}
\address{e-mail: piotr.mizerka@amu.edu.pl}

\begin{abstract}
We compute the characters of real irreducible representations of $\SL_2(q)$, the special linear group on $q$ letters, for an odd prime $q$. Moreover, we give the dimensions of these irreducible representations under the actions of cyclic subgroups of $\SL_2(q)$.
\end{abstract}


\end{frontmatter}

\linenumbers

In what follows, $q$ shall denote an odd prime and we use an abbreviation $(m,n)$ for the greatest common divisor of the integers $m$ and $n$. We call an integer $n$ a \emph{quadratic residue} modulo $q$ if there exists an integer $r$ with $r^2\equiv n\Mod{q}$.
\section{Real irreducible representations of $\SL_2(q)$}
\begin{lemma}
\emph{The conjugacy classes of $\SL_2(q)$ are given by representatives $\mathbf{1},z,c,d,zc,zd,a^l,b^m$, where, for some $\nu\in\mathbb{F}_q^x$ with $\langle\nu\rangle=\mathbb{F}_q^x$,
\begin{center}
\begin{tabular}{ll*{1}{c}r}
$\mathbf{1}=\begin{pmatrix}1&0\\0&1\end{pmatrix}$     
& $z=\begin{pmatrix}-1&0\\0&-1\end{pmatrix}$ 
&$c=\begin{pmatrix}1&0\\1&1\end{pmatrix}$ \\
$d=\begin{pmatrix}1&0\\\nu&1\end{pmatrix}$     
&$a=\begin{pmatrix}\nu&0\\0&\nu^{-1}\end{pmatrix}$ 
&$b$ \\
\end{tabular}
\end{center}
where $b$ is any element of order $q+1$ and $1\leq l\leq (q-3)/2$, $1\leq m\leq (q-1)/2$. The conjugacy classes representatives have the following orders,
\begin{center}
\begin{tabular}{ll*{1}{c}r}
$|\mathbf{1}|=1$     
& $|z|=2$ 
&$|c|=|d|=q$ \\
$|zc|=|zd|=2q$     
&$|a^l|=\frac{q-1}{(l,q-1)}$ 
&$|b^m|=\frac{q+1}{(m,q+1)}$ \\
\end{tabular}
\end{center}}
\end{lemma}
\begin{proof}
The first part of the proof is precisely the statement of \cite[Theorem 38.1]{Dornhoff1971}. We focus now on the representatives order. Obviously, $|\mathbf{1}|=1$ and $|z|=2$. For $c$ and $d$, we have,
\begin{center}
\begin{tabular}{ll*{1}{c}r}
$c^k=\begin{pmatrix}1&0\\k&1\end{pmatrix}$
&$d^k=\begin{pmatrix}1&0\\k\nu&1\end{pmatrix}$ ,    \\
\end{tabular}
\end{center}
hence $|c|=|d|=q$. Now, $zc=\begin{pmatrix}-1&0\\-1&-1\end{pmatrix}$ and $(zc)^k=\begin{pmatrix}(-1)^l&0\\(-1)^k\cdot k&(-1)^k\end{pmatrix}$, therefore $|zc|=2q$. For $zd$, we have \begin{align*}
zd=\begin{pmatrix}-1&0\\-\nu&-1\end{pmatrix}\text{ and } (zd)^k=\begin{pmatrix}(-1)^l&0\\(-1)^k\cdot k\nu&(-1)^k\end{pmatrix}\text{, so }|zd|=2q
\end{align*}
In general situation, if $g\in G$ is of order $n$, then $g^l$ is of order $n/(n,l)$. Since $a$ is of order $q-1$, and $b$ is of order $q+1$, we have $|a^l|=(q-1)/(l,q-1)$ and $|b^m|=(q+1)/(m,q+1)$.
\end{proof}
\begin{proposition}\label{proposition:1}
\emph{Let $p$ be an odd prime and $a\in\{1,\ldots,p-1\}$. Then $a$ is a quadratic residue modulo $p$ if and only if its multiplicative inverse modulo $p$, $a^{-1}$, is a quadratic residue modulo $p$ as well.} 
\end{proposition}
\begin{proof}
Let $r_1$ and $r_2$ be the remainders modulo $p$ of $a^{(p-1)/2}$ and  $(a^{-1})^{(p-1)/2}$ respectively and assume that $a$ is a quadratic residue modulo $p$. Using the Euler's criterion, \cite[pp. 164 Euler's criterion]{Stillwell2003}, we conclude then that $r_1=1$. But 
$$
r_1r_2\equiv a^{(p-1)/2}(a^{-1})^{(p-1)/2}\equiv(aa^{-1})^{(p-1)/2}\equiv 1\Mod{p}
$$
Since $r_2\in\{-1,1\}$, by Euler's criterion, and $r_1=1$, this means that $r_2=1$ and $a^{-1}$ must be a quadratic residue modulo $p$. The converse implication is analogous.
\end{proof}
\begin{proposition}\label{proposition:2}
\emph{$\mathbf{1}^2,z^2\in(\mathbf{1})$. If $2$ is a quadratic residue modulo $q$, then $c^2,(zc)^2\in(c)$ and $d^2,(zd)^2\in(d)$. Otherwise, $c^2,(zc)^2\in(d)$ and $d^2,(zd)^2\in(c)$.}
\end{proposition}
\begin{proof}
Obviously, $\mathbf{1}^2,z^2\in(\mathbf{1})$. Let us establish when $c^2\in(d)$. We have to find $\begin{pmatrix}a&b\\c&d
\end{pmatrix}\in\SL_2(q)$ with
$$
c^2\begin{pmatrix}a&b\\c&d
\end{pmatrix}=\begin{pmatrix}a&b\\c&d
\end{pmatrix}d
$$
Hence we must have 
\begin{align*}
\begin{pmatrix}1&0\\2&1
\end{pmatrix}\begin{pmatrix}a&b\\c&d
\end{pmatrix}=&\begin{pmatrix}a&b\\c&d
\end{pmatrix}\begin{pmatrix}1&0\\\nu&1
\end{pmatrix}\\
\Leftrightarrow \begin{pmatrix}a&b\\2a+c&2b+d
\end{pmatrix}=&\begin{pmatrix}a+\nu b&b\\c+\nu d&d
\end{pmatrix}\Leftrightarrow b=0, d=2\nu^{-1}a
\end{align*}
Taking $a,b,c,d$ as above, we have
$$
\det\begin{pmatrix}
a&b\\c&d
\end{pmatrix}=\det\begin{pmatrix}a&0\\c&2\nu^{-1}a
\end{pmatrix}=2\nu^{-1}a^2
$$
Hence, we can find $a$ with $\det\begin{pmatrix}
a&b\\c&d
\end{pmatrix}=1$ if and only if $(2\nu^{-1})^{-1}$ is a quadratic residue modulo $q$. Therefore, 
$$
((2\nu^{-1})^{-1})^{\frac{q-1}{2}}\equiv 1\Mod{q}\Leftrightarrow \nu^{\frac{q-1}{2}}(2^{-1})^{\frac{q-1}{2}}\equiv 1\Mod{q}
$$
Since $\langle \nu\rangle=\mathbb{F}_q^x$, it follows that $\nu^{(q-1)/2}\equiv -1 (q)$. Therefore, we can find an appropriate $a$ iff
$$
(2^{-1})^{\frac{q-1}{2}}\equiv -1\Mod{q},
$$
which holds iff $2^{-1}$ is not a quadratic residue modulo $q$. This, however, is true whenever $2$ is not a quadratic residue modulo $q$ (conlcusion from Proposition \ref{proposition:1}). On the other hand, $(c)$ and $(d)$ are the only conjugacy classes with elements of order $q$, which is the order of $c^2$ as well. Thus, if $c^2$ does not belong to one of them, it has to belong to the other. Hence, if $2$ is not a quadratic residue modulo $q$, then $c^2\in(d)$, otherwise $c^c\in(c)$. 

Now, consider $d^2=\begin{pmatrix}
1&0\\2\nu&1
\end{pmatrix}$. It belongs to $(c)$ if we can find $\begin{pmatrix}
a&b\\c&d
\end{pmatrix}\in\SL_2(q)$ such that 
\begin{align*}
d^2\begin{pmatrix}
a&b\\c&d
\end{pmatrix}=&\begin{pmatrix}
a&b\\c&d
\end{pmatrix}c\\
\Leftrightarrow \begin{pmatrix}
a&b\\2\nu a+c&2\nu b+d
\end{pmatrix}=&\begin{pmatrix}
a+b&b\\c+d&d
\end{pmatrix}\Leftrightarrow b=0, d=2\nu
a\end{align*}
We must thus have
$
1=\det\begin{pmatrix}
a&b\\c&d
\end{pmatrix}=\begin{pmatrix}
a&0\\c&2\nu a
\end{pmatrix}=2\nu a^2
$. Proceeding analogously as before, we infer that $d^2\in(c)$ if $2$ is not a quadratic residue modulo $q$ and $d^2\in(d)$ otherwise.

Since $(zc)^2=c^2$ and $(zd)^2=d^2$, we infer that $(zc)^2\in(c)$ and $(zd)^2\in (d)$ if $2$ is a quadratic residue modulo $q$ and $(zc)^2\in(d)$ and $(zd)^2\in(c)$ otherwise.  
\end{proof}
\begin{proposition}
\emph{We have
$$(c)\cup(c)^{-1}=\begin{cases} (c)\text{ if } q\equiv 1\Mod{4} \\ (c)\cup(d) \text{ if }q\equiv 3\Mod{4}\end{cases} $$
and
$$(d)\cup(d)^{-1}=\begin{cases} (d) \text{   if }q\equiv 1\Mod{4} \\ (c)\cup(d) \text{ if }q\equiv 3\Mod{4}\end{cases}$$
}
\end{proposition}
\begin{proof}
We have $c^{-1}=\begin{pmatrix}
1&0\\-1&1
\end{pmatrix}$. We want to verify when $c$ is conjugate to $c^{-1}$. This is the case when we can find a matrix $\begin{pmatrix}
a&b\\c&d
\end{pmatrix}\in\SL_2(q)$ such that 
\begin{align*}
c^{-1}\begin{pmatrix}
a&b\\c&d
\end{pmatrix}=&\begin{pmatrix}
a&b\\c&d
\end{pmatrix}c\Leftrightarrow\begin{pmatrix}
1&0\\-1&1
\end{pmatrix}\begin{pmatrix}
a&b\\c&d
\end{pmatrix}=\begin{pmatrix}
a&b\\c&d
\end{pmatrix}\begin{pmatrix}
1&0\\1&1
\end{pmatrix}\\
\Leftrightarrow\begin{pmatrix}
a&b\\-a+c&-b+d
\end{pmatrix}=&\begin{pmatrix}
a+b&b\\
c+d&d
\end{pmatrix}\Leftrightarrow d=-a, b=0
\end{align*}
Thus, $1=\det\begin{pmatrix}
a&b\\c&d
\end{pmatrix}=\det\begin{pmatrix}
a&0\\-a+c&-a
\end{pmatrix}=-a^2=1$
Hence, we can find an appropriate $a$ iff $-1$ is a quadratic residue modulo $q$, that is when $q\equiv 1\Mod{4}$.

We proceed similarly for $d$,
\begin{align*}
(d)=(d^{-1})\Leftrightarrow d^{-1}\begin{pmatrix}
a&b\\c&d
\end{pmatrix}=&\begin{pmatrix}
a&b\\c&d
\end{pmatrix}d
\Leftrightarrow\begin{pmatrix}
1&0\\-\nu&1
\end{pmatrix}\begin{pmatrix}
a&b\\c&d
\end{pmatrix}=\begin{pmatrix}
a&b\\c&d
\end{pmatrix}\begin{pmatrix}
1&0\\\nu&1
\end{pmatrix}\\
\Leftrightarrow\begin{pmatrix}
a&b\\-\nu a+c&-\nu b+d
\end{pmatrix}=&\begin{pmatrix}
a+\nu b&b\\
c+\nu d&d
\end{pmatrix}
\Leftrightarrow d=-a, b=0
\end{align*}
and
$$
1=\det\begin{pmatrix}
a&b\\c&d
\end{pmatrix}=\det\begin{pmatrix}
a&0\\-\nu a+c&-a
\end{pmatrix}=-a^2
$$
which is possible whenever $-1$ is a quadratic residue modulo $q$, i. e. when $q\equiv 1\Mod{4}$. Thus, $d$ is conjugate to $d^{-1}$ if and only if $q\equiv 1\Mod{4}$.  
\end{proof}
\begin{proposition}
\emph{We have \begin{align*}
(zc)\cup(zc)^{-1}=\begin{cases} (zc) \text{ if }q\equiv 1\Mod{4} \\ (zc)\cup(zd) \text{ if }q\equiv 3\Mod{4}\end{cases}
\end{align*}
and
\begin{align*}
(zd)\cup(zd)^{-1}=\begin{cases} (zd) \text{ if }q\equiv 1\Mod{4} \\ (zc)\cup(zd) \text{ if }q\equiv 3\Mod{4}\end{cases}
\end{align*}
}
\end{proposition}
\begin{proof}
The proof is analogous to the proof of the previous proposition.
\end{proof}
For a group $G$, let us define the subset containing the conjugacy class of an element $g\in G$ and the conjugacy class of its inverse, $(g^{-1})$, to be the \emph{real conjugacy class} of the element $g\in G$. The following lemma characterizes the real conjugacy classes of $\SL_2(q)$, 
\begin{lemma}\label{lemma:2}
\emph{If $q\equiv 1\Mod{4}$, then all the conjugacy classes of $\SL_2(q)$ constitute the real conjugacy classes (i. e. any element is conjugate to its inverse). Otherwise, we have $q+2$ real conjugacy classes, $(\mathbf{1}), (z), (c)\cup(d), (zc)\cup(zd),(a^l),(b^m)$ for $1\leq l\leq (q-3)/2$ and $1\leq m\leq (q-1)/2$. }
\end{lemma}
\begin{proof}
By \cite[Theorem 38.1]{Dornhoff1971}, the following table is the complex character table of $\SL_2(q)$,

\begin{center}
\begin{tabular}{l|l*{4}{c}r}
            &  $\mathbf{1}$& $z$ & $c$ & $d$ & $a^l$  & $b^m$ \\
\hline
$\mathds{1}$ & 1 & 1 & 1 & 1 & 1 & 1  \\
$\psi$   & $q$ & $q$ & $0$ & $0$ & $1$ & $-1 $ \\
$\chi_i$   & $q+1$ & $(-1)^i(q+1)$ & $1$ & $1$ & $\nu_{q-1}^{il}$ &$ 0$  \\
$\theta_j$    & $q-1$ & $(-1)^j(q-1)$ & $-1$ & $-1$ &$ 0$ & $-\nu_{q+1}^{jm}$  \\
$\xi_1$     & $\frac{q+1}{2}$ & $\frac{\epsilon(q+1)}{2}$ & $\frac{1+\sqrt{\epsilon q}}{2}$ & $\frac{1-\sqrt{\epsilon q}}{2}$ & $(-1)^l$ & 0 \\
$\xi_2$     & $\frac{q+1}{2}$ & $\frac{\epsilon(q+1)}{2}$ & $ \frac{1-\sqrt{\epsilon q}}{2} $&$ \frac{1+\sqrt{\epsilon q}}{2} $& $(-1)^l $& $0$ \\
$\eta_1$    &$ \frac{q-1}{2} $&$ -\frac{\epsilon(q-1)}{2} $&$ \frac{-1+\sqrt{\epsilon q}}{2} $&$ \frac{-1-\sqrt{\epsilon q}}{2} $& $0$ & $(-1)^{m+1}$ \\
$\eta_2$     &$ \frac{q-1}{2} $&$ -\frac{\epsilon(q-1)}{2} $&$ \frac{-1-\sqrt{\epsilon q}}{2} $&$ \frac{-1+\sqrt{\epsilon q}}{2} $& $0$ & $(-1)^{m+1}$ 
\end{tabular}
\end{center}
where $1\leq i,l\leq (q-3)/2$, $1\leq j,m\leq(q-1)/2$, $\epsilon=(-1)^{(q-1)/2}$, $\zeta_r=\exp{(2\pi i/r)}$, $\nu_r^s=\zeta_r^s+\zeta_r^{-s}$ and $\chi(zc)=\frac{\chi(z)}{\chi(1)}\cdot\chi(c)$ and $\chi(zd)=\frac{\chi(z)}{\chi(1)}\cdot\chi(d)$ for any complex irreducible character $\chi$. Therefore, the only complex irreducible characters taking nonreal values can be $\xi_1,\xi_2,\eta_1,\eta_2$. If it is the case, then all of those characters take some non-real value. This happens precisely in the case $q\equiv 3\Mod{4}$ (i. e. $\epsilon=-1)$. Consider the matrix $P$ such that $\overline{X}=PX$, where $X$ is the character table of $\SL_2(q)$ and $\overline{X}$ denotes the conjugate matrix to $X$. By \cite[23.1 Theorem]{Liebeck2001}, it follows that $P$ is a permutation matrix with trace equal to the number of conjugacy classes which elements are conjugate to their inverses. In our case,
\begin{center}
\begin{tabular}{ll*{1}{c}r}
$\tr P=\begin{cases} q+4 \text{ if }q\equiv 1\Mod{4} \\ q \text{ if }q\equiv 3\Mod{4}\end{cases} $

\end{tabular}
\end{center}
Hence, if $q\equiv 1\Mod{4}$, the assertion follows. If $q\equiv 3\Mod{4}$, then, by the previous two facts, $(c)\cup(d)$ and $(zc)\cup(zd)$ are two different real conjugacy classes. This, in connection with the fact that we have precisely $\tr P$ conjugacy classes with elements conjugate to their own inverses, shows that the remaining conjugacy classes have this property (we have $q+4$ conjugacy classes and $(c),(d),(zc),(zd)$ are the only classes which elements are not conjugate to their inverses).  
\end{proof} 
\begin{corollary}\label{corollary:2}
\emph{In the sequence of conjugacy classes,
$
(a^2),(a^4),\ldots,(a^{q-3}),
$
there are precisely two classes $(a^l)$ for each even $2\leq l\leq (q-3)/2$ and the class $(z)$, only in the case $q\equiv 1 \Mod{4}$. Similarly, in the sequence $
(b^2),(b^4),\ldots,(b^{q-1}),
$, there are precisely two classes $(b^m)$ for each even $2\leq m\leq (q-1)/2$ and the class $(z)$, only in the case $q\equiv 3 \Mod{4}$ }
\end{corollary}
\begin{proof}
Basing on Lemma \ref{lemma:2}, $(a^l)=((a^l)^{-1})$ and $(b^m)=((b^l)^{-1})$ for any $1\leq l\leq (q-3)/2$ and $1\leq m\leq(q-1)/2$. Since $a$ and $b$ are of orders $q-1$ and $q+1$ respectively, it follows that $(a^l)=(a^{q-1-l})$ and $(b^m)=(b^{q+1-m})$. Since $q-1$ and $q+1$ are even, the assertion concerning classes different than $(z)$ follows from the above correspondences. Now, we only have to show that, if $q\equiv 1\Mod{4}$, then $a^{(q-1)/2}=z$ and, if $q\equiv 3\Mod{4}$, $b^{(q+1)/2}=z$. Obviously, $a^{(q-1)/2}$ and $b^{(q+1)/2}$ are elements of order $2$ in the appropriate cases. We show that the only element of order $2$ in $\SL_2(q)$ is $z$. For this, assume that some $\begin{pmatrix}
a&b\\c&d
\end{pmatrix}$ is of order $2$. Then 
$$\begin{pmatrix}
1&0\\0&1
\end{pmatrix}=\begin{pmatrix}
a&b\\c&d
\end{pmatrix}^2=\begin{pmatrix}
a^2+bc&b(a+d)\\c(a+d)&d^2+bc
\end{pmatrix}$$
We have two possible cases. First possibility is that $d=-a$. But then $a^2+bc=1$. On the other hand, 
$$1=\det\begin{pmatrix}
a&b\\c&d
\end{pmatrix}=ad-bc=-a^2-bc,$$ which is a contradiction. Hence, we end up in the second case, namely $b=c=0$. But then $a^2=1$, so for $\begin{pmatrix}
a&b\\c&d
\end{pmatrix}$ to have order $2$, we must have $a=-1$ and $\begin{pmatrix}
a&b\\c&d
\end{pmatrix}=z$ 
\end{proof}
\begin{corollary}\label{corollary:3}
\emph{Let $\chi$ be an irreducible complex character of $\SL_2(q)$. Then, the Frobenius-Schur indicator of $\chi$ is given by
\begin{align*}
\iota(\chi)=\frac{1}{q^3-q}\cdot\Big(2\chi(1)&+(q^2+q)\chi(z)+(q^2-1)(\chi(c)+\chi(d))\\
&+2q\Big((q+1)\sum_{l=1}^{[(q-3)/4]}\chi(a^{2l})+(q-1)\sum_{m=1}^{[(q-1)/4]}\chi(b^{2m})\Big)\Big)
\end{align*}
for $q\equiv 1\Mod{4}$ and
\begin{align*}
\iota(\chi)=\frac{1}{q^3-q}\cdot\Big(2\chi(1)&+(q^2-q)\chi(z)+(q^2-1)(\chi(c)+\chi(d))\\
&+2q\Big((q+1)\sum_{l=1}^{[(q-3)/4]}\chi(a^{2l})+(q-1)\sum_{m=1}^{[(q-1)/4]}\chi(b^{2m})\Big)\Big)
\end{align*}
for $q\equiv 3\Mod{4}$.}
\end{corollary}
\begin{proof}
From Proposition \ref{proposition:2}, we know that $\mathbf{1}^2,z^2\in(\mathbf{1})$ and $c^2,(zc)^2\in(d)$, $d^2,(zd)^2\in(c)$ for $2$ not a quadratic residue modulo $q$ and $c^2,(zc)^2\in(c)$, $d^2,(zd)^2\in(d)$ for $2$ a quadratic residue modulo $q$. By \cite[Theorem 38.1]{Dornhoff1971}, we know that the conjugacy classes sizes of $\SL_2(q)$ are as follows, $|(\mathbf{1})|=|(z)|=1$, $|(c)|=|(d)|=|(zc)|=|(zd)|=(q^2-1)/2$, $|(a^l)|=q(q+1)$ and $|(b^m)|=q(q-1)$ for $1\leq l\leq(q-3)/2$ and $1\leq m\leq(q-1)/2$. Therefore, in both cases, the sum of the summands of the indicator coming from the classes $(\mathbf{1}),(z),(c),(d),(zc)$ and $(zd)$ is 
\begin{align*}
s_1=&\chi(\mathbf{1}^2)+\chi(z^2)+\frac{1}{2}(q^2-1)(\chi(c^2)+\chi(d^2)+\chi((zc)^2)+\chi((zd)^2))\\
=&2\chi(\mathbf{1})+(q^2-1)(\chi(c)+\chi(d))
\end{align*}
Now, by Corollary \ref{corollary:2}, the summands coming from $(a^l)$s and $(b^m)$s sum up to 
$$
s_2=2q\Big((q+1)\sum_{l=1}^{[(q-3)/4]}\chi(a^{2l})+(q-1)\sum_{m=1}^{[(q-1)/4]}\chi(b^{2m})\Big)+q(q+1)\chi(z)
$$
for $q\equiv 1\Mod{4}$ and to
$$
s_2=2q\Big((q+1)\sum_{l=1}^{[(q-3)/4]}\chi(a^{2l})+(q-1)\sum_{m=1}^{[(q-1)/4]}\chi(b^{2m})\Big)+q(q-1)\chi(z)
$$
for $q\equiv 3\Mod{4}$. Summing up $s_1$ and $s_2$, we get the desired formulas for $\iota(\chi)$. 
\end{proof}
We are ready now to figure out the real irreducible characters of $\SL_2(q)$,
\begin{lemma}\label{lemma:5}
\emph{Using the notations from \cite[Theorem 38.1]{Dornhoff1971}, the characters of real irreducible representations of $\SL_2(q)$ are given by $\mathbf{1}, \psi, \chi_{2i},2\chi_{2i'+1},\theta_{2j},2\theta_{2j'+1}$
for $1\leq 2i,2i'+1\leq (q-3)/2$, $1\leq 2j,2j'+1\leq(q-1)/2$ and by $\xi_1,\xi_2,2\eta_1,2\eta_2$ for $q\equiv 1\Mod{4}$ and by $2\operatorname{Re}\xi_1=\xi_1+\xi_2$ and $2\Rex\eta_1=\eta_1+\eta_2$ for $q\equiv 3\Mod{4}$.}
\end{lemma}
\begin{proof}
By \cite[pp. 108]{Serre1977}, we know that if $\chi$ is a complex irreducible character , then there are three possibilities:
\begin{enumerate}
    \item $\iota(\chi)=1$ if $\chi$ is a character of some real irreducible representation
    \item $\iota(\chi)=0$ if $2\Rex\chi=\chi+\overline{\chi}$ is a character of some real irreducible representation (in this case $\chi$ takes some non-real values)
    \item $\iota(\chi)=-1$ if $2\chi$ is a character of some real irreducible representation (int this case $\chi$ takes all real values).
\end{enumerate}
Moreover, every character of some real irreducible representation is obtained in that way. Hence, we only have to show that $\iota(\mathbf{1})=\iota(\psi)=\iota(\chi_{2i})=\iota(\theta_{2j})=1$, $\iota(\chi_{2i+1})=\iota(\theta_{2j+1})=-1$ and $\iota(\xi_1)=\iota(\xi_2)=1$, $\iota(\eta_1)=\iota(\eta_2)=-1$ in case $q\equiv 1 \Mod{4}$ and $\iota(\xi_1)=\iota(\xi_2)=\iota(\eta_1)=\iota(\eta_2)=0$ in case $q\equiv 3 \Mod{4}$. The relation $\iota(\mathbf{1})=1$ is obvious, since $\mathbf{1}$ is the trivial character. In what follows, we use the data from the character table of $\SL_2(q)$.

Thus, by Corollary \ref{corollary:2}, we have
$$
\iota(\psi)=\frac{1}{q^3-q}\cdot\Big(2q+(q^2+q)q+2q\Big((q+1)\cdot\frac{q-5}{4}-(q-1)\cdot\frac{q-1}{4}\Big)\Big)=1
$$
when $q\equiv 1\Mod{4}$ and 
$$
\iota(\psi)=\frac{1}{q^3-q}\cdot\Big(2q+(q^2-q)q+2q\Big((q+1)\cdot\frac{q-3}{4}-(q-1)\cdot\frac{q-3}{4}\Big)\Big)=1
$$
for $q\equiv 3 \Mod{4}$. Notice that from the complex character table of $\SL_2(q)$ follows that $\iota(\xi_1)=\iota(\xi_2)$ and $\iota(\eta_1)=\iota(\eta_2)$ since $\xi_1$ is conjugate (as a complex vector) to $\xi_2$ and $\eta_1$ is conjugate to $\eta_2$. Therefore, if $q\equiv 1 \Mod{4}$,
\begin{align*}
\iota(\eta_1)=\iota(\eta_2)=\frac{1}{q^3-q}\cdot\Big(q-1-&\frac{1}{2}(q^2+q)(q-1)-(q^2-1)\\
+&2q\cdot\Big(-(q-1)\cdot\frac{q-1}{4}\Big)\Big)=-1
\end{align*}
\begin{align*}
\iota(\xi_1)=\iota(\xi_2)=\frac{1}{q^3-q}\cdot\Big(q+1+&\frac{1}{2}(q^2+q)(q+1)+(q^2-1)\\
+&2q\cdot\Big((q+1)\cdot\frac{q-5}{4}\Big)\Big)=1
\end{align*}
and, if $q\equiv 3 \Mod{4}$,
\begin{align*}
\iota(\eta_1)=\iota(\eta_2)=\frac{1}{q^3-q}\cdot\Big(q-1-&\frac{1}{2}(q^2-q)(q-1)-(q^2-1)\\
+&2q\cdot\Big(-(q-1)\cdot\frac{q-3}{4}\Big)\Big)=0
\end{align*}
\begin{align*}
\iota(\xi_1)=\iota(\xi_2)=\frac{1}{q^3-q}\cdot\Big(q+1-&\frac{1}{2}(q^2-q)(q+1)+(q^2-1)\\
+&2q\cdot\Big((q+1)\cdot\frac{q-3}{4}\Big)\Big)=0
\end{align*}
Now, notice that $\chi_i$s and $\theta_j$s are real-valued, since $\nu_r^s=\zeta_r^s+\zeta_r^{-s}$ is real. Therefore, the indicators of these characters are either $1$ or $-1$. We show that $\iota(\chi_{2i})-\iota(\chi_{2i'+1})$ and $\iota(\theta_{2j})-\iota(\chi_{2j'+1})$ are positive what would imply $\iota(\chi_{2i})=\iota(\theta_{2j})=1$ and $\iota(\chi_{2i+1})=\iota(\theta_{2j+1})=-1$ (since $\iota(\chi_i),\iota(\theta_j)\in\{-1,1\}$). Assume $q\equiv 1 \Mod{4}$. Then
\begin{align*}
(q^3-q)(\iota(\chi_{2i})-\iota(\chi_{2i'+1}))=&2(q^2+q)(q+1)+2q(q+1)\sum_{l=1}^{[(q-3)/4]}(\nu_{q-1}^{4il}-\nu_{q-1}^{(4i'+2)l})\\
=&2q(q+1)\Big(q+1+\sum_{l=1}^{[(q-3)/4]}(\nu_{q-1}^{4il}-\nu_{q-1}^{(4i'+2)l})\Big)
\end{align*}
Since $|\nu_r^{s}|\leq 2$ and that the equality holds iff $\frac{r}{2}|s$ for even $r$ and $r|s$ for odd $r$, it follows that
$$
(q^3-q)(\iota(\chi_{2i})-\iota(\chi_{2i'+1}))\geq 2q(q+1)(q+1-4\cdot\frac{q-3}{4})=2q(q+1)\cdot 4>0
$$
Similarly,
\begin{align*}
(q^3-q)(\iota(\theta_{2j})-\iota(\theta_{2j'+1}))=&2(q^2+q)(q-1)\\
+&2q(q-1)\sum_{m=1}^{[(q-1)/4]}(-\nu_{q+1}^{4jm}-\nu_{q+1}^{(4j'+2)m})\\
=&2q(q-1)\Big(q+1+\sum_{m=1}^{[(q-3)/4]}(\nu_{q+1}^{4jm}-\nu_{q+1}^{(4j'+2)m})\Big)\\
\geq& 2q(q-1)\Big(q+1-4\cdot\frac{q-1}{4}\Big)=2q(q-1)\cdot 2>0
\end{align*}
In the case $q\equiv 3 \Mod{4}$, we have
\begin{align*}
(q^3-q)(\iota(\chi_{2i})-\iota(\chi_{2i'+1}))=&2q(q+1)\Big(q-1+\sum_{l=1}^{[(q-3)/4]}(\nu_{q-1}^{4il}-\nu_{q-1}^{(4i'+2)l})\Big)\\
> &2q(q+1)\Big(q-1-4\cdot\frac{q-3}{4}\Big)=2q(q+1)\cdot 2>0
\end{align*}
and
\begin{align*}
(q^3-q)(\iota(\theta_{2j})-\iota(\theta_{2j'+1}))=&2q(q-1)\Big(q-1+\sum_{m=1}^{[(q-1)/4]}(-\nu_{q+1}^{4jm}-\nu_{q+1}^{(4j'+2)m})\Big)\\
>& 2q(q+1)\Big(q-1-4\cdot\frac{q-1}{4}\Big)=0
\end{align*}
\end{proof}
\section{Fixed point dimensions for cyclic subgroups}
Since subgroups' fixed point subspaces' dimensions depend only on the conjugacy classes of subgroups, it suffices to find the dimensions of cyclic subgroups generated by conjugacy classes representatives. Obviously, $\dim V^{\langle \mathbf{1}\rangle}=\dim V$ for any irreducible $\mathbb{R}G$-module $V$. The folllowing tables are the real irreducible character tables derived from Lemma \ref{lemma:5} and the complex character table of $\SL_2(q)$,

\begin{center}
\begin{tabular}{l|l*{4}{c}r}
            &  $\mathbf{1}$& $z$ & $c$ & $d$ & $a^l$  & $b^m$ \\
\hline
$\mathds{1}$ & 1 & 1 & 1 & 1 & 1 & 1  \\
$\psi$   & $q$ & $q$ & $0$ & $0$ & $1$ & $-1 $ \\
$\chi_{2i}$   & $q+1$ & $q+1$ & $1$ & $1$ & $\nu_{q-1}^{2il}$ &$ 0$  \\
$2\chi_{2i+1}$   & $2q+2$ & $-2q-2$ & $2$ & $2$ & $2\nu_{q-1}^{(2i+1)l}$ &$ 0$  \\
$\theta_{2j}$    & $q-1$ & $(-1)^j(q-1)$ & $-1$ & $-1$ &$ 0$ & $-\nu_{q+1}^{2jm}$  \\
$2\theta_{2j+1}$    & $2q-2$ & $-2q+2$ & $-2$ & $-2$ &$ 0$ & $-2\nu_{q+1}^{(2j+1)m}$ 
\end{tabular}
\end{center}
and, if $q\equiv 1\Mod{4}$, then
\begin{center}
\begin{tabular}{l|l*{4}{c}r}
            &  $\mathbf{1}$& $z$ & $c$ & $d$ & $a^l$  & $b^m$ \\
\hline
$\xi_1$     & $\frac{1}{2}(q+1)$ & $\frac{1}{2}\epsilon(q+1)$ & $\frac{1}{2}(1+\sqrt{\epsilon q})$ & $\frac{1}{2}(1-\sqrt{\epsilon q})$ & $(-1)^l$ & 0 \\
$\xi_2$     & $\frac{1}{2}(q+1)$ & $\frac{1}{2}\epsilon(q+1)$ & $ \frac{1}{2}(1-\sqrt{\epsilon q}) $&$ \frac{1}{2}(1+\sqrt{\epsilon q}) $& $(-1)^l $& $0$ \\
$2\eta_1$    &$ q-1 $&$ 1-q $&$ -1+\sqrt{\epsilon q} $&$ -1-\sqrt{\epsilon q} $& $0$ & $2\cdot(-1)^{m+1}$ \\
$\eta_2$     &$ q-1 $&$ 1-q $&$ -1-\sqrt{\epsilon q} $&$ -1+\sqrt{\epsilon q} $& $0$ & $2\cdot(-1)^{m+1}$ 
\end{tabular}
\end{center}
while, when $q\equiv 3\Mod{4}$, then
\begin{center}
\begin{tabular}{l|l*{4}{c}r}
            &  $\mathbf{1}$& $z$ & $c$ & $d$ & $a^l$  & $b^m$ \\
\hline
$2\Rex\xi_1$     & $q+1$ & $-q-1$ & $1$ & $1$ & $2\cdot(-1)^l$ & 0 \\
$2\Rex\xi_2$     & $q-1$ & $q-1$ & $ -1 $&$-1 $& $0$& $2\cdot(-1)^{m+1}$
\end{tabular}
\end{center}
where $1\leq i,l\leq (q-3)/2$, $1\leq j,m\leq(q-1)/2$, $\epsilon=(-1)^{(q-1)/2}$, $\zeta_r=\exp{2\pi i/r}$, $\nu_r^s=\zeta_r^s+\zeta_r^{-s}$ and $\chi(zc)=\frac{\chi(z)}{\chi(1)}\cdot\chi(c)$ and $\chi(zd)=\frac{\chi(z)}{\chi(1)}\cdot\chi(d)$ for any real irreducible character $\chi$.
\begin{lemma}
\emph{$\dim\psi^{\langle z\rangle}=q$, $\dim\chi_{2i}^{\langle z\rangle}=q+1$, $\dim\theta_{2j}^{\langle z\rangle}=q-1$, 
$\dim(2\chi_{2i+1})^{\langle z\rangle}=\dim(2\theta_{2j+1})^{\langle z\rangle}=0$ and, if $q\equiv 1 \Mod{4}$,
$$
\dim\xi_i^{\langle z\rangle}=(q+1)/2, \dim(2\eta_i)^{\langle z\rangle}=0
$$
and, if $q\equiv 3 \Mod{4}$,
$$
\dim(2\Rex\xi_1)^{\langle z\rangle}=0, \dim(2\Rex\eta_1)^{\langle z\rangle}=q-1
$$
Moreover, for $H=\langle c\rangle$ or $\langle d\rangle$, we have
$$
\dim\psi^H=1, \dim\chi_{2i}^H=2,\dim(2\chi_{2i+1})^H=4, \dim\theta_{2j}^H=\dim(2\theta_{2j+1})^H=0
$$
and, if $q\equiv 1 \Mod{4}$, $\dim\xi_i^H=1$, $\dim(2\eta_i)^H=0$ for $i=1,2$. Whereas, if $q\equiv 3 \Mod{4}$, $\dim(2\Rex\xi_1)^H=2$ and $\dim(2\Rex\eta_1)^H=0$.}
\end{lemma}
\begin{proof}
We have $\dim\psi^{\langle z\rangle}=\frac{1}{2}(q+q)=q$, $\dim\chi_{2i}^{\langle z\rangle}=\frac{1}{2}((q+1)+(q+1))= q+1$, $\dim(2\chi_{2i+1})^{\langle z\rangle}=\frac{1}{2}((2q+2)-2q-2)=0$, $\dim\theta_{2j}^{\langle z\rangle}=\frac{1}{2}((q-1)+(q-1))=q-1$, $\dim(2\theta_{2j+1})^{\langle z\rangle}=\frac{1}{2}((2q-2)-2q+2)=0$ and, if $q\equiv 1 \Mod{4}$, then $\dim\xi_i^{\langle z\rangle}=\frac{1}{2}(\frac{q+1}{2}+\frac{q+1}{2})=\frac{q+1}{2}$, $\dim(2\eta_i)^{\langle z\rangle}=\frac{1}{2}((q-1)+(1-q))=0$ for $i=1,2$, while, if $q\equiv 3 \Mod{4}$, $\dim(2\Rex\xi_1)^{\langle z\rangle}=\frac{1}{2}((q+1)-q-1)=0$, $\dim(2\Rex\eta_1)^{\langle z\rangle}=\frac{1}{2}((q-1)+(q-1))=q-1$. Let us focus now on $\langle c\rangle$ and $\langle d\rangle$. Since $\langle c\rangle$ and $\langle d\rangle$ are Sylow subgroups of $\SL_2(q)$ of order $q$, it follows from the Sylow Theorem that they are conjugate and $\dim V^{\langle d\rangle}=\dim V^{\langle c\rangle}$ for any $\mathbb{R}G$-module $V$. Now, notice that 
\begin{eqnarray}\label{eqarray:12}
\dim V^{\langle c\rangle}=\frac{1}{q}\cdot\Big(\chi_{\raisebox{-.5ex}{$\scriptstyle V$}}(1)+\sum_{k=1}^{q-1}\chi_{\raisebox{-.5ex}{$\scriptstyle V$}}(c^k)\Big)
\end{eqnarray}
Since $(c^k)\in\{(c),(d)\}$, we conclude from the shape of real character table that $\dim\psi^{\langle c\rangle}=\frac{1}{q}(q+(q-1)\cdot 0)=1$, $\dim\chi_{2i}^{\langle c\rangle}=\frac{1}{q}((q+1)+(q-1)\cdot 1)=2$, $\dim(2\chi_{2i+1})^{\langle c\rangle}=\frac{1}{q}((2q+2)+(q-1)\cdot 2)=4$, $\dim\theta_{2j}^{\langle c\rangle}=\frac{1}{q}((q-1)+(q-1)\cdot(-1))=0$, $\dim(2\theta_{2j+1})^{\langle c\rangle}=\frac{1}{q}((2q-2)+(q-1)\cdot(-2))=0$ and, if $q\equiv 3 \Mod{4}$, then $\dim(2\Rex\xi_1)^{\langle c\rangle}=\frac{1}{q}((q+1)+(q-1)\cdot 1)=2$, $\dim(2\Rex\eta_1)^{\langle c\rangle}=\frac{1}{q}((q-1)+(q-1)\cdot(-1))=0$.

For the case $q\equiv 1 \Mod{4}$, we have to show a little bit more since it is not the case that $\xi_i$ and $\eta_i$ are equal on $c$ and $d$. We know that $c^k=\begin{pmatrix}
1&0\\k&1
\end{pmatrix}$. We establish when $(c^k)=(c)$ and when $(c^k)=(d)$. We have $(c^k)=(c)$ if and only if we can find a matrix $\begin{pmatrix}
a&b\\c&d
\end{pmatrix}\in\SL_2(q)$ such that
$$
\begin{pmatrix}
1&0\\k&1
\end{pmatrix}\begin{pmatrix}
a&b\\c&d
\end{pmatrix}=\begin{pmatrix}
a&b\\c&d
\end{pmatrix}\begin{pmatrix}
1&0\\1&1
\end{pmatrix}\Leftrightarrow\begin{pmatrix}
a&b\\ka+c&kb+d
\end{pmatrix}=\begin{pmatrix}
a+b&b\\c+d&d
\end{pmatrix}
$$
that is, when $b=0$ and $d=ka$. Thus, $$1=\det\begin{pmatrix}
a&b\\c&d
\end{pmatrix}=\det\begin{pmatrix}
a&0\\c&ka
\end{pmatrix}=ka^2,$$
\end{proof} which holds if and only if $k$ is a quadratic residue modulo $q$. However, by \cite[Exercise 9.3.3]{Stillwell2003}, we know that in the set $\{1,\ldots,q-1\}$ we have precisely $(q-1)/2$ quadratic residues modulo $q$ and $(q-1)/2$ non-residues. Therefore we may count each of $\chi_{\raisebox{-.5ex}{$\scriptstyle V$}}(c)$ and $\chi_{\raisebox{-.5ex}{$\scriptstyle V$}}(d)$ exactly $(q-1)/2$ times when computing (\ref{eqarray:12}). Hence, we have $\dim\xi_1^{\langle c\rangle}=\frac{1}{q}(\frac{q+1}{2}+\frac{q-1}{2}(\frac{1+\sqrt{q}}{2}+\frac{1-\sqrt{q}}{2}))=1$, $\dim\xi_2^{\langle c\rangle}=\frac{1}{q}(\frac{q+1}{2}+\frac{q-1}{2}(\frac{1-\sqrt{q}}{2}+\frac{1+\sqrt{q}}{2}))=1$, $\dim(2\eta_1)^{\langle c\rangle}=\frac{1}{q}((q-1)+\frac{q-1}{2}(-1+\sqrt{q}-1-\sqrt{q}))=0$ and $\dim(2\eta_2)^{\langle c\rangle}=\frac{1}{q}((q-1)+\frac{q-1}{2}(-1-\sqrt{q}-1+\sqrt{q}))=0$
\begin{proposition}
\emph{If $H\leq \SL_2(q)$ and $|H|=2q$, then $H=\langle zc\rangle$ or $H=\langle zd\rangle$.}
\end{proposition}
\begin{proof}
Assume $H\leq \SL_2(q)$ is of order $2q$. It follows by the Sylow Theorem that cyclic subgroups of orders $q$ and $2$ are subgroups of $H$. On the other hand, the only subgroups of order $q$ of $\SL_2(q)$ are $\langle c\rangle$ and $\langle d\rangle$ and the only subgroup of order $2$ is $\langle z\rangle$. Thus, $z\in H$ and $c\in H$ or $d\in H$. If $c\in H$, then $zc\in H$, But $|zc|=2q$, so $|H|=\langle zc\rangle$. The case $d\in H$ is analogous - we get then $H=\langle zd\rangle$.
\end{proof}
\begin{lemma}
\emph{Let $H=\langle zc\rangle$ or $\langle zd\rangle$. Then
$$
\dim\psi^H=1, \dim\chi_{2i}^H=2, \dim(2\chi_{2i+1})^H=\dim\theta_{2j}^H=\dim(2\theta_{2j+1})^H=0
$$
and, if $q\equiv 1\Mod{4}$,
$$
\dim\xi_1^H=\dim\xi_2^H=1, \dim(2\eta_1)^H=\dim(2\eta_2)^H=0
$$
while, if $q\equiv 3 \Mod{4}$,
$$
\dim(2\Rex\xi_1)^H=\dim(2\Rex\eta_1)^H=0
$$}
\end{lemma}
\begin{proof}
We prove only the case $H=\langle zc\rangle$ - the other is analogous.

We have
$$
\dim V^{\langle zc\rangle}=\frac{1}{2q}\sum_{k=0}^{2q-1}\chi_{\raisebox{-.5ex}{$\scriptstyle V$}}((zc)^k)=\frac{1}{2q}(\chi_{\raisebox{-.5ex}{$\scriptstyle V$}}(1)+\chi_{\raisebox{-.5ex}{$\scriptstyle V$}}(z)+\sum_{k=1,k\neq q}^{2q-1}\chi_{\raisebox{-.5ex}{$\scriptstyle V$}}((zc)^k))
$$
Hence, since $\chi(zc)=\frac{\chi(z)}{\chi(1)}\chi(c)$,
$$
\dim\psi^{\langle zc\rangle}=\frac{1}{2q}(q+q+0)=1
$$
$$
\dim\chi_{2i}^{\langle zc\rangle}=\frac{1}{2q}((q+1)+(q+1)+2q-2)=2
$$
$$
\dim(2\chi_{2i+1})^{\langle zc\rangle}\leq\dim(2\chi_{2i+1})^{\langle z\rangle}=0\Rightarrow \dim(2\chi_{2i+1})^{\langle zc\rangle}=0
$$
$$
\dim\theta_{2j}^{\langle zc\rangle}=\frac{1}{2q}((q-1)+(q-1)-2q+2)=0
$$
$$
\dim(2\theta_{2j+1})^{\langle zc\rangle}\leq\dim(2\theta_{2j+1})^{\langle z\rangle}=0
$$
Now, assume $q\equiv 1 \Mod{4}$. Notice that
$$
(\xi_1+\xi_2)(zd)=(\xi_1+\xi_2)(zc)=(\xi_1+\xi_2)(d)=(\xi_1+\xi_2)(c)=1
$$
Therefore,
\begin{eqnarray}\label{eqnarray:2}
\dim(\xi_1+\xi_2)^{\langle zc\rangle}=\frac{1}{2q}\cdot\Big(2\cdot\frac{q+1}{2}+2\cdot\frac{q+1}{2}+2q-2\Big)=2
\end{eqnarray}
On the other hand, we know that $\dim\xi_1^{\langle c\rangle}=\dim\xi_2^{\langle c\rangle}=1$ and $\dim\xi_i^{\langle zc\rangle}\leq\dim\xi_i^{\langle c\rangle}=1$ for $i=1,2$. Thus, in order for (\ref{eqnarray:2}) to be satisfied, we must have $\dim\xi_i^{\langle zc\rangle}=1$ for $i=1,2$. Further,
$
\dim(2\eta_i)^{\langle zc\rangle}\leq \dim(2\eta_i)^{\langle c\rangle}=0
$
for $i=1,2$. Now, let $q\equiv 3 \Mod{4}$. We have then $\dim(2\Rex\xi_1)^{\langle zc\rangle}\leq\dim(2\Rex\xi_1)^{\langle z\rangle}=0$ and $\dim(2\Rex\eta_1)^{\langle zc\rangle}=\frac{1}{2q}((q-1)+(q-1)-(2q-2))=0$
\end{proof}
\begin{lemma}
\emph{Let $1\leq l\leq(q-3)/2$ and $1\leq m\leq(q-1)/2$ be such that $2\nmid(q-1)/(q-1,l)$ and $2\nmid(q+1)/(q+1,m)$ and $H=\langle a^l\rangle$ and $K=\langle b^m\rangle$. Then, for any $1\leq 2i,2i+1\leq(q-3)/2$, $1\leq 2j,2j+1\leq(q-1)/2$, we have $\dim\psi^H=(q-1,l)$, $\dim\chi_{2i}^H=\dim\theta_{2j}^H=\dim(2\eta_1)^H=\dim(2\eta_2)^H=\dim(2\Rex\eta_1)^H=(q-1,l)$, $\dim(2\chi_{2i+1})^H=\dim(2\theta_{2j+1})^H=2(q-1,l)$, $\dim\xi_1^H=\dim\xi_2^H=(q-1,l)/2+1$, $\dim(2\Rex\xi_1)^H=(q-1,l)+2$ and $\dim\psi^K=(q+1,m)-1$, $\dim\chi_{2i}^K=\dim\theta_{2j}^K=\dim(2\Rex\xi_1)^K=(q+1,m)$, $\dim(2\chi_{2i+1})^K=\dim(2\theta_{2j+1})^K=2(q+1,m)$, $\dim\xi_1^H=\dim\xi_2^K=(q+1,m)/2$, $\dim(2\eta_1)^K=\dim(2\Rex\eta_1)^K=(q+1,m)-2$.}
\end{lemma}
\begin{proof}
Since $2\neq(q-1)/(q-1,l)$ and $2\neq(q+1)/(q+1,m)$, it follows that $z\notin H,K$. Hence,
$$
\dim\psi^H=\frac{(q-1,l)}{q-1}\cdot\Big(q+\Big(\frac{q-1}{(q-1,l)}-1\Big)\cdot 1\Big)=(q-1,l)
$$
and
$$
\dim\psi^K=\frac{(q+1,m)}{q+1}\cdot\Big(q+\Big(\frac{q+1}{(q+1,m)}-1\Big)\cdot (-1)\Big)=(q+1,m)-1
$$
Since $\nu_r^s=\nu_r^{r-s}$ for $s\leq r$,  $\nu_r^s=\nu_r^{s\Mod{r}}$ and $\sum_{i=1}^{n-1}\zeta_n^{ik}=-1$ if $k\neq 0$ and $n>1$, we have,
\begin{eqnarray*}
\dim\chi_{2i}^H=&\frac{(q-1,l)}{q-1}\cdot\Big(q+1+\sum_{j=1}^{(q-1)/(q-1,l)-1}\nu_{q-1}^{2ijl}\Big)\\
=&\frac{(q-1,l)}{q-1}\cdot\Big(q+1+\sum_{j=1}^{(q-1)/(q-1,l)-1}\nu_{(q-1)/(q-1,l)}^{j\cdot(2il/(q-1,l))}\Big)\\
=&\frac{(q-1,l)}{q-1}\cdot(q+1-2)=(q-1,l)
\end{eqnarray*}
Now,
$$
\dim\chi_{2i}^K=\frac{(q+1,m)}{q+1}\cdot\Big(q+1+\Big(\frac{q+1}{(q+1,m)}-1\Big)\cdot 0\Big)=(q+1,m)
$$
Similarly,
\begin{align*}
\dim(2\chi_{2i+1})^H=&\frac{(q-1,l)}{q-1}\cdot\Big(2q+2+2\cdot\Big(\sum_{j=1}^{(q-1)/(q-1,l)-1}\nu_{q-1}^{(2i+1)jl}\Big)\Big)\\
=&\frac{(q-1,l)}{q-1}\cdot\Big(2q+2+2\sum_{j=1}^{(q-1)/(q-1,l)-1}\nu_{(q-1)/(q-1,l)}^{j\cdot((2i+1)l/(q-1,l))}\Big)\\
=&\frac{(q-1,l)}{q-1}\cdot(2q+2+2\cdot (-2))=2(q-1,l)
\end{align*}
and
$$
\dim(2\chi_{2i+1})^K=\frac{(q+1,m)}{q+1}(2q+2)=2(q+1,m)
$$
Further,
$$
\dim\theta_{2j}^H=\frac{(q-1,l)}{q-1}\cdot(q-1)=(q-1,l)
$$
\begin{align*}
\dim\theta_{2j}^K=&\frac{(q+1,m)}{q+1}\cdot\Big(q-1-\sum_{i=1}^{(q+1)/(q+1,m)}\nu_{q+1}^{2ijm}\Big)\\
=&\frac{(q+1,m)}{q+1}\cdot\Big(q-1-(-2)\Big)=(q+1,m)
\end{align*}
$$
\dim(2\theta_{2j+1})^H=\frac{(q-1,l)}{q-1}\cdot(2q-2)=2(q-1,l)
$$
\begin{align*}
\dim(2\theta_{2j+1})^K=&\frac{(q+1,m)}{q+1}\cdot\Big(2q-2-2\sum_{i=1}^{(q+1)/(q+1,m)}\nu_{q+1}^{i\cdot(2j+1)m}\Big)\\
=&\frac{(q+1,m)}{q+1}\cdot\Big(2q-2-2\cdot(-2)\Big)=2(q+1,m)
\end{align*}
Assume $q\equiv 1 \Mod{4}$. Then, noticing that $2|l$ and $2|m$, we conldude that the characters of $\xi_1$ and $\xi_2$ evaluated on $a^{il}$ are equal $1$ for any $i=1,\ldots,(q-1)/(q-1,l)-1$. This is because $a^{il}=a^k$, where $k=il\Mod{q-1}$, since the order of $a$ is $q-1$. Moreover, if $k>(q-3/2)$, then $a^k$ is conjugate to $a^{q-1-k}$ and $q+1-k\leq(q-3)/2$ since $z\notin H$. If we denote by $k'$ the minimum of $k$ and $q-1-k$, it follows then that the characters of $\xi_1$ and $\xi_2$ are both equal $(-1)^{k'}$ on $a^{il}$. Since $q-1$ and $l$ are even, we conclude that $k'$ is even as well. Therefore, we have
$$
\dim\xi_1^H=\dim\xi_2^H=\frac{(q-1,l)}{q-1}\cdot\Big(\frac{q+1}{2}+\frac{q-1}{(q-1,l)}-1\Big)=\frac{1}{2}(q-1,l)+1
$$
$$
\dim\xi_1^K=\dim\xi_2^K=\frac{(q+1,m)}{q+1}\cdot\frac{q+1}{2}=\frac{1}{2}(q+1,m)
$$
$$
\dim(2\eta_1)^H=\dim(2\eta_2)^H=\frac{(q-1,l)}{q-1}\cdot(q-1)=(q-1,l)
$$
Similarly as before, we notice that the characters of $2\eta_1$ and $2\eta_2$ evaluated on $b^{jm}$ for $j=1,\ldots,(q+1)/(q+1,m)-1$ are equal to $-1$. Hence
\begin{align*}
\dim(2\eta_1)^K=&\dim(2\eta_2)^K=\frac{(q+1,m)}{q+1}\cdot\Big(q-1+2\cdot\Big(-\frac{q+1}{(q+1,m)}+1\Big)\Big)\\
=&(q+1,m)-2
\end{align*}
If $q\equiv 3 \Mod{4}$, then, by the analogous reasoning, the characters of $2\Rex\xi_1$ and $2\Rex\eta_1$ evaluated on $a^{il}$ and $b^{jm+1}$ are equal $1$ and $-1$ respectively for $1\leq i\leq(q-1)/(q-1,l)-1$ and $1\leq j\leq(q+1)/(q+1,m)-1$. Thus
$$
\dim(2\Rex\xi_1)^H=\frac{(q-1,l)}{q-1}\cdot\Big(q+1+2\cdot\Big(\frac{q-1}{(q-1,l)}-1\Big)\Big)=(q-1,l)+2
$$
$$
\dim(2\Rex\xi_1)^K=\frac{(q+1,m)}{q+1}\cdot(q+1)=(q+1,m)
$$
$$
\dim(2\Rex\eta_1)^H=\frac{(q-1,l)}{q-1}\cdot(q-1)=(q-1,l)
$$
$$
\dim(2\Rex\eta_1)^K=\frac{(q+1,m)}{q+1}\cdot\Big(q-1+2\cdot\Big(-\frac{q+1}{(q+1,m)}+1\Big)\Big)=(q+1,m)-2
$$
\end{proof}
\begin{lemma}
\emph{Let $1\leq l\leq(q-3)/2$ and $1\leq m\leq(q-1)/2$ be such that $2|(q-1)/(q-1,l)$ and $2|(q+1)/(q+1,m)$ and $H=\langle a^l\rangle$ and $K=\langle b^m\rangle$. Then, for any $1\leq 2i,2i+1\leq(q-3)/2$, $1\leq 2j,2j+1\leq(q-1)/2$, we have $\dim\psi^H=2(q-1,l)+1$, $\dim\chi_{2i}^H=\dim\theta_{2j}^H=\dim(2\Rex\eta_1)^H=2(q-1,l)$, $\dim(2\chi_{2i+1})^{H}=\dim(2\theta_{2j+1})^{H}=\dim(2\eta_1)^{H,K}=\dim(2\eta_2)^{H}=\dim(2\Rex\xi_1)^{H}=0$, $\dim\xi_1^H=\dim\xi_2^H=(q-1,l)+1$ for $2|l$ and $\dim\xi_1^H=\dim\xi_2^H=(q-1,l)$ for $2\nmid l$ and $\dim\psi^K=2(q+1,m)-1$, $\dim\chi_{2i}^K=\dim\theta_{2j}^K=2(q+1,m)$, $\dim(2\chi_{2i+1})^K=\dim(2\theta_{2j+1})^K=\dim(2\eta_1)^K=\dim(2\eta_2)^K=\dim(2\Rex\xi_1)^K=0$, $\dim\xi_1^K=\dim\xi_2^K=(q+1,m)$, $\dim(2\Rex\eta_1)^K=2(q+1,m)-2$ for $2|m$ and $\dim(2\Rex\eta_1)^K=2(q+1,m)$ for $2\nmid m$.
}
\end{lemma}
\begin{proof}
We use the similar arguments as in the case $(q-1)/(q-1,l)$ and $(q+1)/(q+1,m)$ were odd. In case they are even, we notice that $z\in H,K$. Thus, we have
$$
\dim\psi^H=\frac{(q-1,l)}{q-1}\cdot\Big(q+q+\Big(\frac{q-1}{(q-1,l)}-2\Big)\cdot 1\Big)=2(q-1,l)+1
$$
$$
\dim\psi^K=\frac{(q+1,m)}{q+1}\cdot\Big(q+q+\Big(\frac{q+1}{(q+1,m)}-2\Big)\cdot (-1)\Big)=2(q+1,m)-1
$$
$$
\dim\chi_{2i}^H=\frac{(q-1,l)}{q-1}\Big(q+1+q+1+\Big(\Big(\sum_{j=1}^{(q-1)/(q-1,l)-1}\nu_{q-1}^{2ijl}\Big)-\nu_{q-1}^{2i\cdot(q-1)/(2\cdot(q-1,l))\cdot l}\Big)\Big)
$$
$$
=\frac{(q-1,l)}{q-1}\cdot\Big(2q+2+\Big(\Big(\sum_{j=1}^{(q-1)/(q-1,l)-1}\nu_{q-1}^{2ijl}\Big)-\nu_{(q-1,l)}^{il}\Big)\Big)
$$
$$=\frac{(q-1,l)}{q-1}\cdot(2q+2+(-2-2))=2(q-1,l)
$$
$$
\dim\chi_{2i}^K=\frac{(q+1,m)}{q+1}(q+1+q+1)=2(q+1,m)
$$
$$
\dim\theta_{2j}^H=\frac{(q-1,l)}{q-1}\cdot(q-1+q-1)=2(q-1,l)
$$
\begin{align*}
\dim\theta_{2j}^K=&\frac{(q+1,m)}{q+1}\cdot\Big(q-1+q-1\\
-&\Big(\Big(\sum_{i=1}^{(q+1)/(q+1,m)-1}\nu_{q+1}^{2ijm}\Big)-\nu_{q+1}^{2\cdot(q+1)/(2(q+1,m)\cdot jm}\Big)\Big)\\
=&\frac{(q+1,m)}{q+1}\cdot\Big(2q-2-\Big(\Big(\sum_{i=1}^{(q+1)/(q+1,m)-1}\nu_{q+1}^{2ijm}\Big)-\nu_{(q+1,m)}^{jm}\Big)\Big)\\
=&\frac{(q+1,m)}{q+1}\cdot\Big(2q-2-(-2-2)\Big)=2(q+1,m)
\end{align*}
Since $\langle z\rangle\leq H, K$ and $\dim(2\chi_{2i+1})^{\langle z\rangle}=\dim(2\theta_{2j+1})^{\langle z\rangle}=0$, we get
$$
\dim(2\chi_{2i+1})^H=\dim(2\chi_{2i+1})^K=\dim(2\theta_{2j+1})^H=\dim(2\theta_{2j+1})^K=0
$$
Assume $q\equiv 1 \Mod{4}$. Since $z\in H,K$ and $\dim(2\eta_1)^{\langle z\rangle}=\dim(2\eta_2)^{\langle z\rangle}$, we have
$$
\dim(2\eta_1)^H=\dim(2\eta_1)^K=\dim(2\eta_2)^H=\dim(2\eta_1)^K=0
$$
Let us consider the characters $\xi_1$ and $\xi_2$. From the real character table we see that their fixed point dimensions for subgroups $H$ and $K$ are equal. Let us consider the case $2|l$. Then, as before, the characters of $\xi_1$ and $\xi_2$ evaluated on $a^{il}$ are equal $1$ for $1\leq i\leq (q-1)/(q-1,l)-1$. Thus, in this case,
$$
\dim\xi_1^H=\dim\xi_2^H=\frac{(q-1,l)}{q-1}\cdot\Big(\frac{q+1}{2}+\frac{q+1}{2}+\frac{q-1}{(q-1,l)}-2\Big)=(q-1,l)+1
$$
In the second case, when $2\nmid l$, we observe that the characters $\xi_1$ and $\xi_2$ evaluated on $a^{il}$ are eeual to $1$ when $i$ is odd and to $-1$ when $i$ is even. Thus, the sum $\sum_{i=1,i\neq (q-1)/(2(q-1,l))}^{(q-1)/(q-1,l)-1}\xi_{1,2}(a^{il})$ is equal to $-2$, since $(q-1)/(2(q-1,l))$ is even because $q\equiv 1\Mod{4}$, so
$$
\dim\xi_1^H=\dim\xi_2^H=\frac{(q-1,l)}{q-1}\cdot\Big(\frac{q+1}{2}+\frac{q+1}{2}-2\Big)=(q-1,l)
$$
Moreover, in both cases,
$$
\dim\xi_1^K=\dim\xi_2^K=\frac{(q+1,m)}{q+1}\Big(\frac{q+1}{2}+\frac{q+1}{2}\Big)=(q+1,m)
$$
If $q\equiv 3 \Mod{4}$, then, using the same arguments as in the case $q\equiv 1 \Mod{4}$, we get $\dim(2\Rex\xi_1)^H=\dim(2\Rex\xi_1)^K=0$. Further,
$$
\dim(2\Rex\eta_1)^H=\frac{(q-1,l)}{q-1}\cdot(q-1+q-1)=2(q-1,l)
$$
and, if $2|m$,
\begin{eqnarray*}
\dim(2\Rex\eta_1)^K=&\frac{(q+1,m)}{q+1}\cdot\Big(q-1+q-1+2\cdot\Big(-\frac{q+1}{(q+1,m)}+2\Big)\Big)\\
=&2(q+1,m)-2
\end{eqnarray*}
whereas, if $2\nmid m$, then, since $q\equiv 3\Mod{4}$, we have $2| (q+1)/(2(q+1,m))$ and $\sum_{j=1,j\neq (q+1)/(2(q+1,m))}^{(q+1)/(q+1,m)-1}(2\Rex\eta_1)(b^{jm})=4$. Thus,
$$
\dim(2\Rex\eta_1)^K=\frac{(q+1,m)}{q+1}\cdot(q-1+q-1+4)=2(q+1,m)
$$
\end{proof}
Let us summarize in the following tables the dimensions of cyclic subgroups of $\SL_2(q)$ for real irreducible representations,
\begin{center}
\begin{tabular}{l|l*{3}{c}r}
            &  $\mathbf{1}$& $\langle z\rangle$ & $\langle c\rangle=\langle d\rangle$ & $\langle zc\rangle=\langle zd\rangle$ \\
\hline
$\mathds{1}$ & 1 & 1 & 1 & 1  \\
$\psi$   & $q$ & $q$ & $1$ & $1$  \\
$\chi_{2i}$   & $q+1$ & $q+1$ & $2$ & $2$   \\
$2\chi_{2i+1}$   & $2q+2$ & $0$ & $4$ & $0$   \\
$\theta_{2j}$    & $q-1$ & $q-1$ & $0$ & $0$  \\
$2\theta_{2j+1}$    & $2q-2$ & $0$ & $0$ & $0$   \\
$\xi_1=\xi_2$     & $\frac{1}{2}(q+1)$ & $\frac{1}{2}(q+1)$ & $1$ & $1$ \\
$2\eta_1=2\eta_2$    &$ q-1 $&$ 0 $&$0 $&$ 0$ \\
$2\Rex\xi_1$     & $q+1$ & $0$ & $2$ & $0$  \\
$2\Rex\eta_1$     & $q-1$ & $q-1$ & $ 0 $&$0$ 
\end{tabular}
\end{center}
and, if we put $H=\langle a^l\rangle$, $K=\langle b^m\rangle$ for $2\nmid\frac{q-1}{(q-1,l)}, \frac{q+1}{(q+1,m)}$, $H'=\langle a^l\rangle$, $K'=\langle b^m\rangle$ for $2 |\frac{q-1}{(q-1,l)}, \frac{q+1}{(q+1,m)}$ and $2|l,m$ and $H''=\langle a^l\rangle$, $K''=\langle b^m\rangle$ for $2|\frac{q-1}{(q-1,l)}, \frac{q+1}{(q+1,m)}$ and $2\nmid l,m$, then
\begin{center}
\begin{tabular}{l|{c}r}
            &  $H$& $K$ \\
\hline
$\mathds{1}$ & 1 & 1 \\
$\psi$   & $(q-1,l)+1$ & $(q+1,m)-1$  \\
$\chi_{2i}$   & $(q-1,l)$ & $(q+1,m)$  \\
$2\chi_{2i+1}$   & $2(q-1,l)$ & $2(q+1,m)$  \\
$\theta_{2j}$    & $(q-1,l)$ & $(q+1,m)$  \\
$2\theta_{2j+1}$    & $2(q-1,l)$ & $2(q+1,m)$ \\
$\xi_1=\xi_2$     & $\frac{(q-1,l)}{2}+1$ & $\frac{(q+1,m)}{2}$  \\
$2\eta_1=2\eta_2$    &$ (q-1,l) $&$ (q+1,m)-2 $ \\
$2\Rex\xi_1$     & $(q-1,l)+2$ & $(q+1,m)$   \\
$2\Rex\eta_1$     & $(q-1,l)$ & $(q+1,m)-2$  
\end{tabular}
\end{center}
and
\begin{center}
\begin{tabular}{l|l*{2}{c}r}
            &  $H'$ & $K'$ &$H''$&$K''$\\
\hline
$\mathds{1}$ & 1 & 1 &1&1 \\
$\psi$    & $2(q-1,l)+1$ & $2(q+1,m)-1$&2(q-1,l)+1&2(q+1,m)-1  \\
$\chi_{2i}$    & $2(q-1,l)$ & $2(q+1,m)$&  $2(q-1,l)$ & $2(q+1,m)$ \\
$2\chi_{2i+1}$    & $0$ & $0$&$0$&$0$   \\
$\theta_{2j}$     & $2(q-1,l)$ & $2(q+1,m)$&$2(q-1,l)$&$2(q+1,m)$  \\
$2\theta_{2j+1}$     & $0$ & $0$   &0&0\\
$\xi_1=\xi_2$      & $(q-1,l)+1$ & $(q+1,m)$&$(q-1,l)$&$(q+1,m)$ \\
$2\eta_1=2\eta_2$    &$0 $&$ 0$&$0$&$0$ \\
$2\Rex\xi_1$     & $0$ & $0$&$0$&$0$  \\
$2\Rex\eta_1$      & $ 2(q-1,l) $&$2(q+1,m)-2$&$2(q-1,l)$&$2(q+1,m)$ 
\end{tabular}
\end{center}
where $1\leq i,l\leq (q-3)/2$, $1\leq j,m\leq(q-1)/2$, $\epsilon=(-1)^{(q-1)/2}$, $\zeta_r=\exp{2\pi i/r}$, $\nu_r^s=\zeta_r^s+\zeta_r^{-s}$ and $\chi(zc)=\frac{\chi(z)}{\chi(1)}\cdot\chi(c)$ and $\chi(zd)=\frac{\chi(z)}{\chi(1)}\cdot\chi(d)$ for any real irreducible character $\chi$.



\newpage
\bibliographystyle{model1-num-names}
\bibliography{sample.bib}







\end{document}